\documentclass[a4paper,reqno]{amsart}
\usepackage{iftex}
\ifPDFTeX
   \usepackage[utf8]{inputenc}
   \usepackage[T1]{fontenc}
   \usepackage{lmodern}
   \usepackage{amsmath,amssymb,amsthm}
\else
   \ifXeTeX
     \usepackage{fontspec}
   \else 
     \usepackage{unicode-math}
     \usepackage{luatextra}
   \fi
   \defaultfontfeatures{Ligatures=TeX}
\fi

\usepackage{mathtools}
\usepackage{thmtools}
\usepackage{xparse}

\newcommand{\R}{\mathbb{R}}
\newcommand{\N}{\mathbb{N}}
\usepackage{bbm}
\newcommand{\indicator}[1]{\mathbbm{1}_{#1}}

\renewcommand{\d}{\,d}

\newcommand{\restr}[2]{#1\!_{\restriction #2}}

\DeclarePairedDelimiter\abs{\lvert}{\rvert}%
\DeclarePairedDelimiter\norm{\lVert}{\rVert}%
\makeatletter
\let\oldabs\abs
\def\abs{\@ifstar{\oldabs}{\oldabs*}}
\let\oldnorm\norm
\def\norm{\@ifstar{\oldnorm}{\oldnorm*}}
\makeatother

\newcommand{\fullstop}{\text{ .}}

\NewDocumentCommand{\lcro}{d<>m}{\IfValueTF{#1}{#1}{\left}[ #2 \IfValueTF{#1}{#1}{\right})}

\newtheorem{theorem}{Theorem}
\numberwithin{theorem}{section}
\newtheorem{lemma}[theorem]{Lemma}

\theoremstyle{definition}
\newtheorem{definition}[theorem]{Definition}
\newtheorem{remark}[theorem]{Remark}

\usepackage{csquotes}

\usepackage[scr=pxtx]{mathalfa}
\usepackage{relsize}
\def\fcmp{\mathbin{\raise 0.6ex\hbox{\oalign{\hfil$\scriptscriptstyle \mathrm{o}$\hfil\cr\hfil$\scriptscriptstyle\mathrm{9}$\hfil}}}}
\let\mcmp\fcmp

\usepackage{tikz}
\usepackage{enumitem}
\usepackage[normalem]{ulem}
\setlist[enumerate]{label=(\arabic*)}

\usepackage[hidelinks]{hyperref}
\renewcommand{\d}{\,\mathrm d}

\newcommand{\Prob}{\mathscr P}
\RenewDocumentCommand{\Pr}{d<>m}{\Prob\IfValueTF{#1}{#1}{\!\left}(#2\IfValueTF{#1}{#1}{\right})}
\NewDocumentCommand{\Prp}{d<>m}{\Prob\!_p\IfValueTF{#1}{#1}{\!\left}(#2\IfValueTF{#1}{#1}{\right})}
\NewDocumentCommand{\SubP}{d<>m}{\Prob^\leq\IfValueTF{#1}{#1}{\left}(#2\IfValueTF{#1}{#1}{\right})}
\NewDocumentCommand{\E}{d<>omo}{\mathbb E \IfValueTF{#2}{^{#2}}{} \IfValueTF{#1}{#1}{\left}( #3 \IfValueTF{#4}{ \IfValueTF{#1}{#1}{\middle}| #4}{} \IfValueTF{#1}{#1}{\right})}

\DeclareMathOperator{\dis}{dis}
\newcommand{\disint}[2]{\dis_{#1}^{#2}}
\NewDocumentCommand{\br}{d<>mmm}{\IfValueTF{#1}{#1}{\left}#2 #4 \IfValueTF{#1}{#1}{\right}#3}
\NewDocumentCommand{\pa}{d<>m}{\IfValueTF{#1}{#1}{\left}( #2 \IfValueTF{#1}{#1}{\right})}
\NewDocumentCommand{\push}{d<>m}{\Prob\IfValueTF{#1}{#1}{\!\left}(#2\IfValueTF{#1}{#1}{\right})}
\NewDocumentCommand{\FunP}{d<>mm}{\mathscr F \IfValueTF{#1}{#1}{\left}( #2 \rightsquigarrow #3 \IfValueTF{#1}{#1}{\right})}
\DeclareMathOperator\domain{dom}
\NewDocumentCommand{\dom}{d<>m}{\domain\! \IfValueTF{#1}{#1}{\left}( #2 \IfValueTF{#1}{#1}{\right}) }
\DeclareMathOperator{\Couplings}{Cpl}
\NewDocumentCommand{\Cpl}{d<>mm}{\Couplings \IfValueTF{#1}{#1}{\left}( #2, #3 \IfValueTF{#1}{#1}{\right}) }
\NewDocumentCommand{\set}{d<>mo}{\IfValueTF{#1}{#1}{\left}\{ #2 \IfValueT{#3}{\,\IfValueTF{#1}{#1}{\middle}|\, #3} \IfValueTF{#1}{#1}{\right}\}}

\DeclareMathOperator{\proj}{proj}
\DeclareMathOperator{\undis}{int}
\newcommand{\und}[2]{\undis_{#1}^{#2}}

\newcommand{\ilcomment}[1]{{\fontfamily{cmss}\selectfont{}#1}}

\newcommand{\Xb}{\overline X}

\newcommand\funprod{\setbox0\hbox{$\prod$}\rlap{\hbox to \wd0{\hss$\circ$\hss}}\box0}
\newcommand{\I}{\mathcal I}

\newcommand{\dprod}{\mathbin{\dot\otimes}}
\DeclareMathOperator{\ShOp}{Per}

\let\epsilon\varepsilon
\newcommand{\meet}{\wedge}
\newcommand{\join}{\vee}
\newcommand{\itref}[1]{\ref{#1}}

\newcommand{\marg}[2]{#1_{\restriction #2}}

\let\comment\marginpar

\usepackage[bottom]{footmisc}
\usetikzlibrary{matrix,fit,positioning}
\usepackage{cancel}
\let\unused\tilde

\newcommand{\X}{\mathcal X}
\newcommand{\Y}{\mathcal Y}
\newcommand{\Z}{\mathcal Z}
\newcommand{\mocf}[1]{\omega_{#1}}
\newcommand{\moc}[2]{\omega_{#1}\pa{#2}}
\newcommand{\Sh}[2]{\ShOp\pa{#1,#2}}
\newcommand{\avg}[1]{\undis_{*}^{#1}}
\newcommand{\Wa}{\mathcal W_p}
\NewDocumentCommand{\W}{d<>mm}{\Wa \IfValueTF{#1}{#1}{\left}( #2, #3 \IfValueTF{#1}{#1}{\right}) }
\newcommand{\gobble}[1]{}
\gobble{
to execute: yank and then :@"
:let @n="dd'nP''"
:iabbrev bega \begin{align*}<CR>]]<Esc>k$a
:iabbrev begt \begin{tikzpicture}<CR>]]<Esc>k$a
:iabbrev begp \begin{proof}<CR>]]<Esc>k$a
:iabbrev begl \begin{lemma}<CR>]]<Esc>k$a
:iabbrev beg {<Esc>xxa}<Esc>hhbi\begin{}<Esc>xo]]<Esc>k$a

$\end{lemma}]\end{proof}\end{tikzpicture}\end{align*}
}

\newcommand{\D}{\rho}

\renewcommand{\I}{\mathcal I}
\let\comment\gobble

\let\ilcomment\gobble

\author{Manu Eder}
  \address{Department of Mathematics, University of Vienna, Austria}
  \email{manuel.eder@univie.ac.at}
  \date{\today}

\title{Compactness in Adapted Weak Topologies}

\begin{document}

\begin{abstract}
  Over the years a number of topologies for the set of laws of stochastic processes have been proposed. Building on the weak topology they all aim to capture more accurately the temporal structure of the processes.

  In a parallel paper we show that all of these topologies (i.e.\ the information topology of Hellwig, the nested distance topology of Pflug-Pichler, the extended weak convergence of Aldous and a topology built from Lasalle's notion of a causal transference plan) are equal in finite discrete time.
  Regrettably, the simple characterization of compactness given by Prokhorov's theorem for the weak topology fails to be true in this finer topology. This phenomenon is closely related to the failure of a natural metric for this topology to be complete. For certain problems, a \enquote{fix} consists in passing to the metric completion. Still, it also seems interesting to find out what compact sets look like in the uncompleted space.

  Here we give a characterization of compact sets in this adapted weak topology which is strongly reminiscent of the Arzelà-Ascoli theorem (with a dash of Prokhorov's theorem).
  The tools developed are also useful elsewhere. We give a different proof of the continuity of the conditionally independent gluing map of two measures with one marginal in common and in our companion paper the ideas developed here form the main non-algebraic ingredient in showing that the information topology introduced by Hellwig is equal to the nested weak topology of Pflug-Pichler.
\end{abstract}

\maketitle

\section{Introduction}

In Figure \ref{fig:Adap} we see the possible paths for two different stochastic processes. We'll think of each of the paths drawn as having the same probability $1/2$. The process on the left only branches at final time $2$, while the one on the right already branches at time $1$, but the branches don't move very far apart. The processes on the left and on the right are very close in Wasserstein distance, but their \enquote{information structure} is very different. For the process on the right we already know at time $1$, what is going to happen at time $2$, for the one on the left we don't. 

\begin{figure}
  \label{fig:Adap}
  \includegraphics[trim={0 1cm 0 4.5cm},clip,width=.8\linewidth]{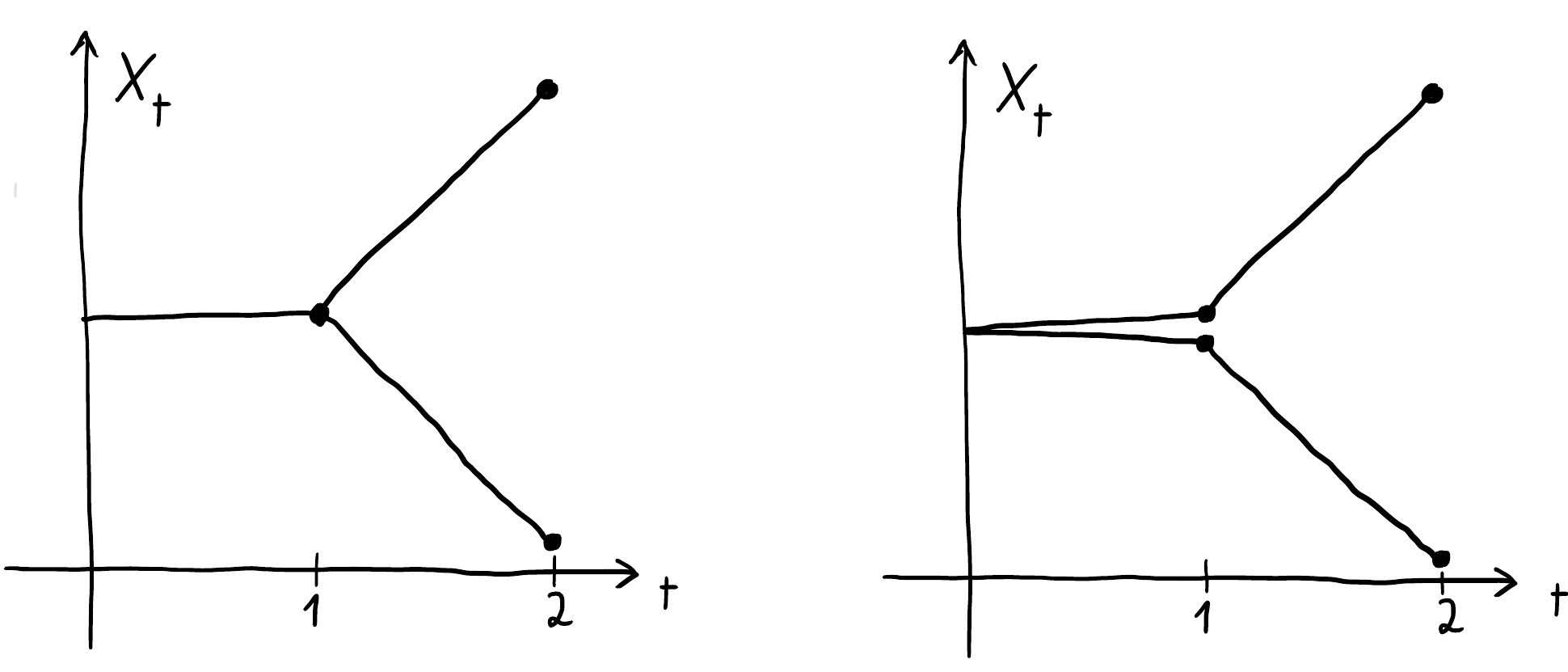}%
  \caption{Two processes which are very close in Wasserstein distance, but whose information structure is very different.}
\end{figure}

A number of authors have introduced topologies and/or metrics which respect this information structure of processes -- topologies for which, in particular, the two processes in Figure \ref{fig:Adap} are \emph{not} \enquote{close} to each other.
These are: Hellwig's information topology \cite{He96}, the nested distance of Pflug, Pichler and co-authors \cite{PfPi12, Pi13,  PfPi14, PfPi15, PfPi16, GlPfPi17} and the extended weak topology of Aldous \cite{Al81}. Lassalle's notion of a causal transference plan, \cite{Las18}, can also be utilized to define a metric by restricting the transference plans in the definition of the Wasserstein metric to be causal and then symmetrizing.
In a parallel paper \cite{AllTopologiesAreEqual} we show that all these topologies are in fact equal in the finite discrete time setting.

Already by looking at the pictures in Figure \ref{fig:Adap} one can see that all of these topologies will necessarily lack a feature which is often very useful -- namely the characterization of compactness by something akin to Prokhorov's theorem.
Let us imagine a sequence of laws of processes $\mu_n$ described by pictures similar to the one on the right, only with the size of the gap at time $1$ going to zero.
We had just decided that, if the topology is to respect the \enquote{information structure} of processes, then the sequence $(\mu_n)_n$ cannot converge to the measure $\mu$ described by the picture on the left.
If the topology is also finer than the weak topology (which is a feature that all of the cited topologies share) then $(\mu_n)_n$, and any of its subsequences have nowhere to converge to.
This is even though $\mu_n$ very much remain bounded in any of the usual senses, so by any fictitious generalization of Prokhorov's theorem to this new topology should be relatively compact.

One \enquote{fix} for this problem, which has already seen some use for example in \cite{BaBePa18,BaPa19}, is to pass to a larger space which (among others) contains an extra element which $(\mu_n)_n$ converges to.
But we are also interested in finding out what the (relatively) compact sets in the original space are.

We now give a rigorous definition of the information topology as introduced by Hellwig, as this is the formulation that it is easiest to work with for the purposes of this paper (see \cite{AllTopologiesAreEqual} for all the equivalent ways of describing this topology) and then state our main theorem, Theorem \ref{thm:relconested}, which gives a characterization of relatively compact sets in the information topology. We would like to emphasize the parallels between this theorem and the theorem of Arzelà-Ascoli describing compact sets in spaces of continuous functions.

Let $\Z$ be a Polish space. In fact, let us fix a compatible complete bounded metric, so that we are viewing $\Z$ as a Polish metric space with a bounded metric $\D_Z$. We are interested in probability measures on $\Z^N$, where $N \in \N$. We denote by $Z_t : \Z^N \rightarrow \Z$ the projection on the $t$-th coordinate, i.e.\ $(Z_t)_t$ is the canonical process on $\Z^N$.

\newcommand{\Law}{\mathcal L}
Building on the idea already alluded to that we want to capture what we may predict about the future evolution of a process from its behaviour up to the current time $t$ we introduce maps
\begin{align*}
  \I_t : \Pr {\Z^N} \rightarrow \Pr {\Z^t \times \Pr {\Z^{N-t}}}
\end{align*}
which send a measure $\mu$ to the joint law of
\begin{align*}
  Z_1, \dots, Z_t, \Law^{\mu}(Z_{t+1},\dots,Z_N | Z_1, \dots, Z_t)
\end{align*}
under $\mu$. $\Law^{\mu}(Z_{t+1},\dots,Z_N | Z_1, \dots, Z_t)$ denotes the conditional law of $Z_{t+1},\dots,Z_N$ given $Z_1, \dots, Z_t$ under $\mu$.

\begin{definition}
  Hellwig's \emph{information topology} on $\Pr {\Z^N}$ is the initial topology w.r.t.\ $\set{\I_t}[1 \leq t < N]$.
\end{definition}

In Definition \ref{def:mocprel} we introduce the central notion used in characterizing relative compactness in the information topology. First we need a little more notation.

For any Polish space $\X$ call $\Pr \X$ the set of probability measures and $\SubP \X$ the set of subprobability measures on $\X$.

\begin{definition}[Modulus of Continuity]
  \label{def:mocprel}
  Let $\X$ and $\Y$ be Polish metric spaces and let $\mu \in \Pr { \X \times \Y }$.
  The \emph{modulus of continuity} $\mocf \mu : \R_+ \rightarrow \R_+$ of $\mu$ is given by
  \begin{align*}
    \moc \mu \delta & := \sup_{\gamma \in \Sh \mu \delta} \D^\Y(\gamma)
  \end{align*}
  where
  \begin{align*}
    \D^\X(\gamma) & := {\textstyle \int} \D_\X(x_1,x_2) \d \gamma(x_1,y_1,x_2,y_2) \text{, } &
    \D^\Y(\gamma) & := {\textstyle \int} \D_\Y(y_1,y_2) \d \gamma(x_1,y_1,x_2,y_2)
  \end{align*}
  and
  \begin{multline*}
    \Sh \mu \delta := \set<\big>{ \gamma \in \SubP { \X \times \Y \times \X \times \Y } }[{ \\ \text{both $\X \times \Y$-marginals of $\gamma$ are $\leq \mu$} \text{ and } \D^\X(\gamma) \leq \delta }]
  \end{multline*}
  is the set of measures describing \enquote{perturbations} of $\mu$ that (on average) shift the $\X$-coordinate by at most $\delta$.
\end{definition}

\begin{remark}
  In the definition of $\Sh \mu \delta$ we might as well have said $\gamma \in \Pr { \X \times \Y \times \X \times \Y }$ instead of $\gamma \in \SubP { \X \times \Y \times \X \times \Y }$ without changing the definition of $\moc \mu \delta$, see Lemma \ref{lem:mocSubPvsPr}. For our purposes the definition given here is more convenient.
\end{remark}

Note that $\Law^{\mu}(Z_{t+1},\dots,Z_N | Z_1, \dots, Z_t)$, being a conditional law, is a function of $Z_1, \dots, Z_t$.
Setting $\X := \Z^t$ and $\Y := \Pr {\Z^{N-t}}$ we see that $\I_t(\mu)$ is probability on $\X \times \Y$ and is concentrated on the graph of a measurable function $\X \rightarrow \Y$.
$\X$ can be equipped with the $\ell^1$-metric $\D_\X((z_i)_i, (z'_i)_i) := \sum_{i=1}^t \D_\Z(z_i,z'_i)$, which is a bounded compatible complete metric and $\Y$ can be equipped with the $1$-Wasserstein metric built from the sum metric on $\Z^{N-t}$, which is a complete metric inducing the usual weak topology on $\Pr {\Z^{N-t}}$.
In the following, when we write $\mocf {\I_t(\mu)}$ this is how we want $\X$ and $\Y$ in the definition of the modulus of continuity to be understood.

\begin{theorem}
  \label{thm:relconested}
  $K \subseteq \Pr {\Z^N}$ is relatively compact in the information topology iff
  \begin{enumerate}
    \item \label{it:weaklycompact}
      \comment{If we want to adapt the whole paper to be both about weak topology and Wasserstein metrics for possibly unbounded metrics then the wording here needs to change:\\ }
      $K$ is relatively compact in the weak topology and
    \item \label{it:kmalgleichgradigstetig}
      $\displaystyle \lim_{\delta \searrow 0} \sup_{\mu \in K} \moc {\I_t(\mu)} \delta = 0$ for all $t \in \set{ 1, \dots, N-1 }$.
  \end{enumerate}
  \comment{The statement is true both for $\I_t = \disint {\Xb^t} {\Xb_{t+1}}$ and for $\I_t = \disint {\Xb^t} {\X_{t+1}} \circ \proj_{\Xb^{t+1}}$.}
\end{theorem}

\section{Properties of the Modulus of Continuity}

We will see in the proof of Theorem \ref{thm:relconested} that if we understand relative compactness in the information topology in the case of two timepoints, there's not much difficulty in passing to the $N$-timepoint case. So we will first focus on the two-timepoint case.
Here the information topology is the toplogy that we get on $\Pr {\Z^2}$ when we embed it into $\Pr {\Z \times \Pr \Z}$ via $\I_1$.
In fact $\Pr {\Z^2}$ with the information topology is homeomorphic to the subspace of $\Pr {\Z \times \Pr \Z}$ whose elements are all probability measures which are concentrated on the graph of a Borel function $\Z \rightarrow \Pr \Z$, equipped with the subspace topology.

So this is the setting in which we will begin studying the problem. We have two Polish metric spaces $\X$ and $\Y$ and we are interested in the relatively compact sets in $\FunP \X \Y \subseteq \Pr {\X \times \Y}$, the space of measures on $\X \times \Y$ which are concentrated on the graph of some Borel function from $\X$ to $\Y$.

\subsection{From \texorpdfstring{$1$}{1}-Wasserstein to \texorpdfstring{$p$}{p}-Wasserstein}
\label{sec:wassersteinp}

At this point we would like to clarify a small detail that we have tried to mostly gloss over up to now.
In the introduction we have been switching between talking about topological spaces and talking about metric spaces.
This was for expositional purposes, because we wanted to show how our results connect to the literature on \enquote{adapted weak topologies}, more specifically the information topology, which has only been defined as a topology -- not a metric -- by Hellwig.
As can be seen from Definition \ref{def:mocprel} of the modulus of continuity, our methods make direct use of a metric.
By choosing a compatible complete bounded metric on $\Z$ (and $\Z^{N-t}$) we get the $1$-Wasserstein metric (or really any $p$-Wasserstein metric) to induce the usual weak topology on $\Pr {\Z^{N-t}}$ and are thus able to recover topological results about the weak topology and the information topology.

\newcommand{\Probp}{\mathscr P_{\!\!p}}
\RenewDocumentCommand{\Pr}{d<>m}{\Probp\IfValueTF{#1}{#1}{\left}(#2\IfValueTF{#1}{#1}{\right})}
\NewDocumentCommand{\PrN}{d<>m}{\Prob\IfValueTF{#1}{#1}{\left}(#2\IfValueTF{#1}{#1}{\right})}
The methods themselves do not rely on the assumption that the metrics are bounded, though.
They work for any Polish metric space and provide statements about compact sets in the topology induced by the $1$-Wasserstein distance.
In fact, they are also easily generalized to $p$-Wasserstein distances for $p \geq 1$.

Therefore, in the sequel let us make the following conventions, which we will be using unless noted otherwise.
$1 \leq p < \infty$ can be chosen now and is kept fixed throughout the paper.
All spaces $\X, \Y, \Z$ etc.\ denoted by calligraphic\comment{check if still true in the end} letters are Polish metric spaces. The metric on $\X$ will be called $\D_{\X}$, etc. If clear from the context we may omit the subscript. For any two Polish metric spaces $\X$ and $\Y$ their product space $\X \times \Y$ will be regarded as a Polish metric space with the metric
\begin{align*}
  \D_{\X \times \Y}\pa<\big>{(x_1,y_1),(x_2,y_2)} := \pa<\big>{ \D_\X(x_1,x_2)^p + \D_\Y(y_1,y_2)^p }^{\frac 1 p} \fullstop
\end{align*}
Note that this construction is associative so that there is no confusion about what the metric on for example $\X \times \Y \times \Z$ should be, as both groupings $(\X \times \Y) \times \Z$ and $\X \times (\Y \times \Z)$ give the same result. So for example the metric on $\Z^t$ is 
\begin{align*}
  \D_{\Z^t}\pa<\big>{(z_i)_i, (z'_i)_i} = \pa<\big>{ {\textstyle \sum_i} \, \D_\Z(z_i,z'_i)^p }^{\frac 1 p} \fullstop
\end{align*}
For any Polish metric space $\X$, $\Pr \X$ will denote the space of probability measures $\mu$ on $\X$ with finite $p$-th moment, i.e.\ satisfying
\begin{align*}
  \int \D_\X(x_0,x)^p \d\mu(x) < \infty
\end{align*}
for any (and therefore all) $x_0$ and will carry the $p$-Wasserstein metric
\begin{align*}
  \D_{\Pr \X} (\mu_1, \mu_2) := \Wa(\mu_1,\mu_2) = \pa{\inf_{\gamma \in \Cpl {\mu_1} {\mu_2}} \int \D_\X(x_1,x_2)^p \d\gamma(x_1,x_2)}^{\frac 1 p}
\end{align*}
where $\Cpl {\mu_1} {\mu_2}$ is the set of couplings between $\mu_1$ and $\mu_2$, i.e.\ the set of measures $\gamma \in \PrN {\X \times \X}$ with first marginal $\mu_1$ and second marginal $\mu_2$.

\RenewDocumentCommand{\FunP}{d<>mm}{\mathscr F_{\!p} \IfValueTF{#1}{#1}{\left}( #2 \rightsquigarrow #3 \IfValueTF{#1}{#1}{\right})}
$\FunP \X \Y$ is the space of $\mu \in \Pr {\X \times \Y}$ which are concentrated on the graph of some Borel function from $\X$ to $\Y$.


We also amend Definition \ref{def:mocprel}.

\begin{definition}[$p$-Modulus of Continuity]
  \label{def:moc}
  Let $\X$ and $\Y$ be Polish metric spaces and let $\mu \in \Pr { \X \times \Y }$.
  The \emph{modulus of continuity} $\mocf \mu : \R_+ \rightarrow \R_+$ of $\mu$ is given by
  \begin{align*}
    \moc \mu \delta & := \sup_{\gamma \in \Sh \mu \delta} \D^\Y(\gamma)
  \end{align*}
  where
  \begin{align*}
    \D^\X(\gamma) & := \pa<\Big>{{\textstyle \int} \D_\X(x_1,x_2)^p \d \gamma(x_1,y_1,x_2,y_2)}^{\frac 1 p} \text{,}\\
    \D^\Y(\gamma) & := \pa<\Big>{{\textstyle \int} \D_\Y(y_1,y_2)^p \d \gamma(x_1,y_1,x_2,y_2)}^{\frac 1 p}
  \end{align*}
  and
  \begin{multline*}
    \Sh \mu \delta := \set<\big>{ \gamma \in \SubP { \X \times \Y \times \X \times \Y } }[{ \\ \text{both $\X \times \Y$-marginals of $\gamma$ are $\leq \mu$} \text{ and } \D^\X(\gamma) \leq \delta }]
  \end{multline*}
\end{definition}

\begin{remark}
  \label{rem:DXprop}
  There are two main properties of $\D^\X$ and $\D^\Y$ that we will be making use of in our proofs.
  The first is that for $r \geq 0$
  \begin{align}
    \label{eq:DXhomog}
    \D^\X(r \gamma) = r^{1/p} \, \D^\X(\gamma) \fullstop
  \end{align}
  The second is that $\D^\X(\gamma)$ is really the $L^p(\gamma)$-norm of $(x_1,y_1,x_2,y_2) \mapsto \D_\X(x_1,x_2)$. If we can decompose this function as a sum of functions or bound it by a sum of function then we may apply the triangle inequality of $L^p(\gamma)$.
\end{remark}

\subsection{Basic properties of the modulus of continuity}

Now we start listing basic properties of $\moc \mu \delta$.

First we show that in the definition of $\moc \mu \delta$ it does not matter whether we talk about probabilities or subprobabilities.

\begin{definition}
  Let $\gamma \in \SubP {\X \times \Y \times \X \times \Y}$. The mirrored version, or \emph{inverse}, $\gamma^{-1}$ of $\gamma$ is the pushforward of $\gamma$ under the map $(x_1,y_1,x_2,y_2) \mapsto (x_2,y_2,x_1,y_1)$.
\end{definition}

\begin{lemma}
  \label{lem:mocSubPvsPr}
  Let $\mu \in \Pr {\X \times \Y}$.
  For any $\gamma' \in \SubP {\X \times \Y \times \X \times \Y}$, both of whose $\X \times \Y$-marginals are $\leq \mu$, there is a $\gamma \in \PrN {\X \times \Y \times \X \times \Y}$ both of whose $\X \times \Y$-marginals are equal to $\mu$, which satisfies $\D^\X(\gamma) = \D^\X(\gamma')$, $\D^\Y(\gamma) = \D^\Y(\gamma')$ and which is symmetric in the sense that $\gamma = \gamma^{-1}$.
\end{lemma}
\begin{proof}
  Given $\gamma' \in \SubP {\X \times \Y \times \X \times \Y}$ we first symmetrize by setting $\gamma_2 := \frac 1 2 \pa{\gamma' + \gamma'^{-1}}$. Because metrics are symmetric, $\D^\X(\gamma_2) = \D^\X(\gamma')$ and $\D^\Y(\gamma_2) = \D^\Y(\gamma')$. Now both the first and the second $\X \times \Y$-marginal of $\gamma_2$ is equal to some measure $\mu' \leq \mu$. If we add the identity coupling of $\mu-\mu'$, i.e.\ the measure $\push{(x,y)\mapsto (x,y,x,y)}\pa{\mu - \mu'}$, to the measure $\gamma_2$ we get a measure $\gamma$ which is still symmetric, still satisfies $\D^\X(\gamma) = \D^\X(\gamma')$, $\D^\Y(\gamma) = \D^\Y(\gamma')$ and which has both marginals equal to $\mu' + (\mu - \mu') = \mu$ and therefore must be a probability measure.
\end{proof}

\begin{lemma}
  $\mocf \mu$ is monotone, i.e.\ $\delta_1 \leq \delta_2$ implies $\moc \mu {\delta_1} \leq \moc \mu {\delta_2}$.
\end{lemma}
\begin{proof}
  Obvious.
\end{proof}

\begin{lemma}
  $\restr{\mocf \mu}{(0,\infty)}$ is continuous.
\end{lemma}
\begin{proof}
  Let $0 < \delta_1 < \delta_2$. Let $\gamma \in \Sh \mu {\delta_2}$.

  By \eqref{eq:DXhomog} we have $r \gamma \in \Sh \mu {\delta_1}$, if we set $r := \pa{\frac {\delta_1} {\delta_2}}^p$.
  So $\moc \mu {\delta_1} \geq \D^\Y(r \gamma) = \frac {\delta_1} {\delta_2} \D^\Y(\gamma)$.
  As $\gamma \in \Sh \mu {\delta_2}$ was arbitrary we have
  \begin{align}
    \label{eq:moccontindelta}
    \moc \mu {\delta_1} \geq \frac {\delta_1} {\delta_2} \moc \mu {\delta_2} \fullstop
  \end{align}
  Let $\delta > 0$, let $|\delta' - \delta| < \varepsilon'$ where $\varepsilon'$ is small enough that both 
  \begin{align*}
    \pa{ 1 - \frac {\delta - \varepsilon'} \delta } \moc \mu \delta & < \varepsilon &
    \pa{ \frac {\delta + \varepsilon'} \delta - 1} \moc \mu \delta & < \varepsilon \fullstop
  \end{align*}
  If $\delta' < \delta$ then subtracting \eqref{eq:moccontindelta} with $\delta_2 = \delta$, $\delta_1 = \delta'$ from $\moc \mu \delta$ we get
  \begin{align*}
    | \moc \mu \delta - \moc \mu {\delta'} | = \moc \mu \delta - \moc \mu {\delta'} \leq \pa{1 - \frac {\delta'} \delta} \moc \mu \delta
  \end{align*}
  If $\delta < \delta'$ then similarly multiplying \eqref{eq:moccontindelta} by $\frac {\delta_2} {\delta_1}$, substituting $\delta_2 = \delta'$, $\delta_1 = \delta$  and subtracting $\moc \mu \delta$ from it we get
  \begin{align*}
    | \moc \mu \delta - \moc \mu {\delta'} | = \moc \mu {\delta'} - \moc \mu \delta \leq \pa{ \frac {\delta'} \delta - 1 } \moc \mu \delta \fullstop
  \end{align*}
\end{proof}

The following lemma shows how the analogy hinted at by calling $\mocf \mu$ the modulus of continuity is to be understood. While the classical modulus of continuity recognizes \emph{continuous functions} $f$ as those for which $\lim_{\delta \searrow 0} \moc f \delta = 0$, our modulus of continuity for measures recognizes \emph{measures concentrated on the graph of a function}.

\begin{lemma}
  \label{lem:FunPCharacterisation}
  Let $\mu \in \Pr { \X \times \Y }$. Then $\mu \in \FunP \X \Y$ iff $\lim_{\delta \searrow 0} \moc \mu \delta = 0$.
\end{lemma}
\begin{proof}
  \comment{I can't immediately figure out if $\mocf \mu$ has to be continuous at $0$ when $\mu \not\in \FunP \X \Y$... My feeling says that maybe somehow applying this result to the disintegration would show that yes, but also maybe not...}

  By monotonicity of $\mocf \mu$, $\lim_{\delta \searrow 0} \moc \mu \delta = 0$ implies $\moc \mu 0 = 0$. We first show that this in turn implies $\mu \in \FunP \X \Y$.
  For any $\mu \in \Pr { \X \times \Y } $ we can always construct the following $\gamma \in \Sh \mu 0 \subseteq \PrN { \X \times \Y \times \X \times \Y} $.
  Let $(\mu_x)_{x \in \X}$ be a disintegration of $\mu$ w.r.t.\ the first coordinate.
  \begin{align*}
    \gamma(f) & := \iiint f(x,y_1,x,y_2) \d \mu_x(y_2) \d \mu_x(y_1) \d \mu(x,\unused y) \\ &\phantom{:}=  \iint f(x,y_1,x,y_2) \d \pa{ \mu_x \otimes \mu_x }(y_1,y_2) \d \mu(x,\unused y)
  \end{align*}
  $\moc \mu 0 = 0$ implies that
  \begin{align*}
    0 = \D^\Y(\gamma)^p = \iint \D_\Y(y_1,y_2)^p \d \pa{ \mu_x \otimes \mu_x }(y_1,y_2) \d \mu(x,\unused y) \fullstop
  \end{align*}
  This means that for $\marg \mu \X$-a.a.\ $x$ we have $\int \D_\Y(y_1,y_2)^p \d \pa{ \mu_x \otimes \mu_x }(y_1,y_2) = 0$. This implies that $\mu_x$ is concentrated on a single point, and there is a measurable map $b$ sending measures concentrated on a single point to that point. $b \circ (x \mapsto \mu_x)$ is then the function on whose graph $\mu$ is concentrated. This concludes the first half of the proof.

  We now show that $\mu \in \FunP \X \Y$ implies $\lim_{\delta \searrow 0} \moc \mu \delta = 0$.
  Let $f: \X \rightarrow \Y$ be a measurable function such that $\int g(x,y) \d \mu(x,y) = \int g(x,f(x)) \d \mu(x, y)$.

  Fix $y_0 \in \Y$, let $\theta > 0$ be such that $g : \X \times \Y \rightarrow [0,1]$, $\mu(g) < \theta$ implies 
  \begin{align}
    \label{eq:abscont}
    \int \D(y_0,y)^p g(x,y) \d \mu(x,y) < \epsilon^{p} \fullstop
  \end{align}
  This is possible because the finite measure which has density $(x,y) \mapsto \D(y_0,y)^p$ w.r.t.\ $\mu$ is absolutely continuous w.r.t.\ to $\mu$.

  Because $\X$ is Polish and $\Y$ is second countable we can apply Lusin's theorem to get a compact set $K \subseteq \X$ such that $\restr f K$ is uniformly continuous and $\marg \mu \X (K^C) < \frac \theta 3$.
  Let $\eta > 0$ be such that for $x_1, x_2 \in K$, $\D(x_1,x_2) < \eta$ implies $\D(f(x_1),f(x_2)) < \epsilon$.
  Let $\delta < (\frac \theta 3)^{1/p} \cdot \eta$.

  Let $\gamma \in \Sh \mu \delta$. $\D^\Y(\gamma)$ is the $L^p(\gamma)$-norm of the function $(x_1,y_1,x_2,y_2) \mapsto \D(y_1,y_2)$ which, setting
  \begin{align*}
    R(x_1,x_2) & := \indicator{(K \times K)^C}(x_1,x_2) + \indicator{K \times K}(x_1,x_2) \, \indicator{\lcro{\eta,\infty}}(\D(x_1,x_2))
  \end{align*}
  we may bound as follows
  \begin{multline*}
    \D(y_1,y_2) = \D(y_1,y_2) \, R(x_1,x_2) + \D(y_1,y_2) \, \indicator{K \times K}(x_1,x_2) \, \indicator{\lcro{0,\eta}}(\D(x_1,x_2))  \\
    \leq \pa<\big>{\D(y_1,y_0) + \D(y_0,y_2)} \, R(x_1,x_2) 
    + \D(y_1,y_2) \, \indicator{K \times K}(x_1,x_2) \, \indicator{\lcro{0,\eta}}(\D(x_1,x_2))
  \end{multline*}

  Using the triangle inequality in $L^p(\gamma)$ and the fact $\mu$ is concentrated on the graph of $f$ we get that
  \begin{multline*}
    \D^\Y(\gamma) \leq 
    \pa<\bigg>{\int \D(y_0,y_1)^p R(x_1,x_2) \d\gamma(x_1,y_1,x_2,y_2)}^{1/p} + \\
    \pa<\bigg>{\int \D(y_0,y_2)^p R(x_1,x_2) \d\gamma(x_1,y_1,x_2,y_2)}^{1/p} + \\
    \pa<\bigg>{\int \D(f(x_1),f(x_2))^p \, \indicator{K \times K}(x_1,x_2) \, \indicator{\lcro{0,\eta}}(\D(x_1,x_2))\d \gamma(x_1,y_1,x_2,y_2) }^{1/p}
  \end{multline*}

  The first two integrals are of the form as in \eqref{eq:abscont} and as
  \begin{align*}
    \theta > \int R(x_1,x_2) \d\gamma(x_1,y_1,x_2,y_2) = \iint R(x_1,x_2) \d\gamma_{x_1,y_1}(x_2,y_2) \d\mu(x_1,y_1)
  \end{align*}
  by the choice of $K$ and because $\gamma \in \Sh \mu \delta$, they can each be bounded by $\epsilon^p$. In the last integral, whenever the integrand is nonzero, $\D(f(x_1),f(x_2)) < \epsilon$ by our choice of $K$ and $\eta$.
  Overall we get
  \begin{align*}
    \D^\Y(\gamma) < 3 \epsilon \fullstop
  \end{align*}
\end{proof}

\subsection{Composition of measures}

In the proof of Lemma \ref{lem:moccontinmu} below we will be \enquote{composing} measures on product spaces to get new measures. A useful intuition may be to think of the operation $\mcmp$ below as a generalization of the composition of functions or relations. From a probabilistic point of view $\gamma \dprod \gamma'$ below should be called the conditionally independent product (at least when both $\gamma$ and $\gamma'$ are probability measures).
\begin{definition}
  \label{def:mcmp}
  For $\gamma \in \Pr {\X \times \Y}$ and $\lambda \in \Pr {\Y \times \Z}$ with $\marg \gamma \Y = \marg {\lambda} \Y$ define
\begin{align*}
  \gamma \dprod \lambda & := f \mapsto \int f(x,y,z) \d \lambda_y(z) \d \gamma(x,y) \\
  \gamma \mcmp \lambda & := f \mapsto \int f(x,z) \d \lambda_y(z) \d \gamma(x,y)
\end{align*}
where $y \mapsto \lambda_y$ is a disintegration of $\lambda$ w.r.t.\ the first variable, that is $\int f \d\lambda = \int \int f(y,z) \d \lambda_y(z) \d \gamma(y,\unused z)$.
\end{definition}
The asymmetry in the definition is only apparent, in the sense that we may as well have disintegrated $\gamma$ instead of $\lambda$, getting the same result.
Both $\dprod$ and $\mcmp$ are associative operations.

\begin{lemma}
  \label{lem:moccontinmu}
  Let $\delta > 0$. Then
  \begin{align*}
    \mu \mapsto \moc \mu \delta 
  \end{align*}
  is continuous on $\Pr{ \X \times \Y}$, i.e.\ in the $p$-Wasserstein metric.
\end{lemma}
\begin{proof}
  Let $\mu, \nu \in \Pr{\X \times \Y}$ and let $\Wa(\mu,\nu) < \epsilon$. We will show that then \eqref{eq:moccontinmuconc} below holds. As both sides of \eqref{eq:moccontinmuconc} converge to $\moc \mu \delta$ as $\epsilon$ goes to $0$ this shows that $\mu \mapsto \moc \mu \delta$ is continuous at $\mu$.

  $\Wa(\mu,\nu) < \epsilon$ implies that there is $\psi \in \Cpl \mu \nu$ s.t.\ $\D^\X(\psi) \join \D^\Y(\psi) < \epsilon$.

  We want to bound $\moc \mu \delta$ in terms of $\moc \nu \delta$, so let $\gamma \in \Sh \mu \delta$ be arbitrary. By Lemma \ref{lem:mocSubPvsPr} we may as well assume that $\gamma$ is a probability measure. Then
{\def\unused#1{{\scriptstyle #1}}
  \begin{multline*}
    \D^\X(\psi \mcmp \gamma \mcmp \psi^{-1}) = \pa{\int \D(x_1,x_4)^p \d \pa{ \psi \mcmp \gamma \mcmp \psi^{-1} }(x_1,\unused{y_1},x_4,\unused{y_4})}^{1/p} \leq \\
    \pa{\int \pa<\big>{\D(x_1,x_2) + \D(x_2,x_3) + \D(x_3,x_4)}^p \d \pa{ \psi \dprod \gamma \dprod \psi^{-1} }(x_1,\unused{y_1},x_2,\unused{y_2},x_3,\unused{y_3},x_4,\unused{y_4})}^{1/p} \leq \\
    \D^\X(\psi) + \D^\X(\gamma) + \D^\X(\psi^{-1}) < \D^\X(\gamma) + 2 \epsilon < \delta + 2 \epsilon \fullstop
  \end{multline*}%
}

  Scaling down, we get that $r \cdot \psi \mcmp \gamma \mcmp \psi^{-1} \in \Sh \nu \delta$, where $r := \pa{\frac \delta {\delta + 2 \epsilon}}^p$. By definition of $\moc \nu \delta$
  \begin{align*}
    \D^\Y\pa{r \cdot \psi \mcmp \gamma \mcmp \psi^{-1}} \leq \moc \nu \delta
    \intertext{or}
    \D^\Y\pa{\psi \mcmp \gamma \mcmp \psi^{-1}} \leq \pa{ 1 + \frac {2 \epsilon} \delta } \moc \nu \delta
  \end{align*}
  We can bound $\D^\Y(\gamma)$ in terms of $\D^\Y\pa{\psi \mcmp \gamma \mcmp \psi^{-1}}$:
{\def\unused{}
  \begin{multline*}
    \D^\Y(\gamma) = \pa{\int \D(y_2,y_3)^p \d \pa{\psi \dprod \gamma \dprod \psi^{-1}}(\unused{x_1},y_1,\unused{x_2},y_2,\unused{x_3},y_3,\unused{x_4},y_4)}^{1/p} \leq \\
    \pa{\int \pa<\big>{\D(y_2,y_1) + \D(y_1,y_4) + \D(y_4,y_3)}^p \d \pa{\psi \dprod \gamma \dprod \psi^{-1}}(\unused{x_1},y_1,\unused{x_2},y_2,\unused{x_3},y_3,\unused{x_4},y_4)}^{1/p} \leq \\
    \D^\Y(\psi) + \D^\Y\pa{ \psi \mcmp \gamma \mcmp \psi^{-1} } + \D^\Y\pa{\psi^{-1}} < \\
    \pa{ 1 + \frac {2 \epsilon} \delta } \moc \nu \delta + 2 \epsilon
  \end{multline*}%
}
  As $\gamma$ was arbitrary this implies
  \begin{align*}
    \moc \mu \delta < \pa{ 1 + \frac {2 \epsilon} \delta } \moc \nu \delta + 2 \epsilon \fullstop
  \end{align*}
  Rearranging terms gives the left side of \eqref{eq:moccontinmuconc}, while repeating the argument with the roles of $\mu$ and $\nu$ swapped gives the right side of \eqref{eq:moccontinmuconc}.

  \begin{align}
    \label{eq:moccontinmuconc}
    \frac { \moc \mu \delta - 2 \epsilon } { 1 + \frac {2 \epsilon} \delta } < \moc \nu \delta < \pa{ 1 + \frac {2 \epsilon} \delta } \moc \mu \delta + 2 \epsilon
  \end{align}
\end{proof}

\begin{theorem}
  \label{thm:relco}
  Let $K \subseteq \FunP \X \Y$. Then $K$ is relatively compact in $\FunP \X \Y$ (equipped with the $p$-Wasserstein metric) iff
  \begin{enumerate}
    \item \label{it:relcoinPr}
      $K$ is relatively compact in $\Pr { \X \times \Y }$ (equipped with the $p$-Wasserstein metric) and
    \item \label{it:gleichgradigstetig}
      $ \displaystyle\lim_{\delta \searrow 0} \sup_{\mu \in K} \moc \mu \delta = 0 $.
  \end{enumerate}
\end{theorem}
\begin{proof}
  We first show that \itref{it:relcoinPr} and \itref{it:gleichgradigstetig} together imply that $K$ is relatively compact in $\FunP \X \Y$.

  To that end we show that every sequence in $K$ has a subsequence which converges to a point in $\FunP \X \Y$. So let $(\mu_n)_n$ be a sequence in $K$. By \itref{it:relcoinPr} there is a subsequence $(\mu_{n_k})_k$ which converges to a point $\mu \in \Pr { \X \times \Y }$.
  By continuity of the modulus of continuity in its measure argument, i.e.\ by Lemma \ref{lem:moccontinmu}, and by assumption \itref{it:gleichgradigstetig}
  \begin{align*}
    \lim_{\delta \searrow 0} \moc \mu \delta = \lim_{\delta \searrow 0} \lim_{k \to \infty} \moc {\mu_{n_k}} \delta \leq \lim_{\delta \searrow 0} \sup_{\nu \in K} \moc {\nu} \delta = 0 \fullstop
  \end{align*}
  By Lemma \ref{lem:FunPCharacterisation} this implies $\mu \in \FunP \X \Y$.

  The implication from \enquote{$K$ relatively compact in $\FunP \X \Y$} to \itref{it:relcoinPr} is trivial. To show that \enquote{$K$ relatively compact in $\FunP \X \Y$} implies \itref{it:gleichgradigstetig} we show its contrapositive.

  So assume that \itref{it:gleichgradigstetig} is false. Then there is an $\epsilon$ and for all $n \in \N$ a measure $\mu_n \in K$ with $\moc {\mu_n} {\frac 1 n} \geq \epsilon$. Because $\delta \mapsto \moc {\mu_n} \delta$ is monotone this means that $\restr{\pa{\mocf {\mu_n}}}{\lcro{\frac 1 n,\infty}} \geq \epsilon$. For any subsequence $(\mu_{n_k})_k$ of $(\mu_n)_n$ which converges to some $\mu \in \Pr { \X \times \Y }$ we have again by Lemma \ref{lem:moccontinmu}
  \begin{align*}
    \lim_{\delta \searrow 0} \moc \mu \delta = \lim_{\delta \searrow 0} \lim_{k \to \infty} \moc {\mu_{n_k}} \delta \geq \varepsilon
    \fullstop
  \end{align*}
  This means that $\mu \notin \FunP \X \Y$.
\end{proof}

\ilcomment{
  I haven't spent too much time thinking about the interconnection between the different versions of the modulus of continuity for different $p$-Wasserstein metrics, $p \in \set{0} \cup \lcro{1,\infty}$. It seems to me that in this theorem we could have used the \enquote{$0$-Wasserstein}-version of the modulus of continuity as well...
  The reason I'm not really inclined to do it is because I suspect that somehow the theory is simplest when you just stick with one $p$ and it feels \enquote{ugly} to mix different versions and to introduce this arbitrary element of turning the metric into a bounded metric by some procedure.
  But this thing, that it seems to me here we could just always be using the $0$-Wasserstein-version might be hinting that really there isn't much of the $p$ in those parts of the modulus of continuity that we care about (i.e.\ in the asymptotic behaviour at $0$).
  I'm not terribly motivated to really figure it out, just wanted to mention that this is something I noticed in case someone else (or me at some later time) thinks that this is really, really important to figure out.
}

\section{Relative Compactness in the Nested Weak Topology}
\label{sec:relconested}

We are now ready to prove Theorem \ref{thm:relconested}. We restate it below as Theorem \ref{thm:relconested2}, generalizing from the weak topology to the one induced by the $p$-Wasserstein metric.

\begin{definition}
  The $\Wa$-information topology is the initial topology with respect to the maps $\I_t$, $t \in \set{1, \dots, N-1}$, with the target spaces $\Pr {\Z^t \times \Pr {\Z^{N-t}}}$ equipped with the topology which arises when we use the $p$-Wasserstein metric throughout as per our convention introduced at the beginning of Section \ref{sec:wassersteinp}.
\end{definition}

For this to make sense we need to check that $\I_t(\mu) \in \Pr { \Z^t \times \Pr {\Z^{N-t}}}$ i.e.\ that
\begin{align*}
  \int \D(\hat z_0, \hat z)^p \d\pa{\I_t(\mu)}(\hat z) < \infty
\end{align*}
for some $\hat z_0 \in \Z^t \times \Pr {\Z^{N-t}}$. Let $z_0 \in \Z^t$, $z_0' \in \Z^{N-t}$, and set $\hat z_0 := (z_0, \delta_{z_0'})$. Then one easily checks
\begin{align*}
  \int \D(\hat z_0, \hat z)^p \d\pa{\I_t(\mu)}(\hat z) = \int \D((z_0,z_0'), z) \d \mu(z) < \infty \fullstop
\end{align*}

At this point we would also like to add another minor generalization, which is to allow the process to take its values in different spaces for different times. Let $\Z_t$, $t \in \set{1, \dots, N}$ be Polish spaces. The role of $\Z^N$ is now played by $\prod_{t=1}^N \Z_t$ and the process at time $t$ takes values in $\Z_t$. We introduce the shorthands
\newcommand{\overbar}[1]{\mkern 1.5mu\overline{\mkern-1.5mu#1\mkern-1.5mu}\mkern 1.5mu}
\renewcommand{\Xb}{\overbar\Z}
\begin{align*}
  \Xb_s^t & := \prod_{i=s}^t \Z_i  &
  \Xb     & := \Xb_1^N &
  \Xb^t   & := \Xb_1^t &
  \Xb_t   & := \Xb_t^N \fullstop
\end{align*}

\begin{theorem}
  \label{thm:relconested2}
  $K \subseteq \Pr {\Xb}$ is relatively compact in the $\Wa$-information topology iff
  \begin{enumerate}
    \item \label{it:weaklycompact2}
      $K$ is relatively compact in $\Pr {\Xb}$, i.e.\ in the topology induced by the $p$-Wasserstein metric and
    \item \label{it:kmalgleichgradigstetig2}
      $\displaystyle \lim_{\delta \searrow 0} \sup_{\mu \in K} \moc {\I_t(\mu)} \delta = 0$ for all $t \in \set{ 1, \dots, N-1 }$.
  \end{enumerate}
\end{theorem}

\begin{proof}[Proof of Theorem \ref{thm:relconested2} (and therefore also Theorem \ref{thm:relconested})]
  That $K$ being relatively compact in the $\Wa$-information topology implies \itref{it:weaklycompact} and \itref{it:kmalgleichgradigstetig} is clear because when $\Pr \Xb$ is equipped with the $\Wa$-information topology both the identity to $\Pr \Xb$ equipped with the usual topology and all of the $\I_t$ are continuous, therefore map relatively compact sets to relatively compact sets.

  To show the reverse implication we need to show that
  \begin{enumerate}[label=(\alph*)]
    \item \label{it:PPandPcompactness} $K$ relatively compact in the usual topology on $\Pr \Xb$ implies that $\I_t\br[]{K}$ is relatively compact in $\Pr {\Xb^t \times \Pr {\Xb_{t+1}}}$.
    \item \label{it:IXclosed} $\I[\Pr \Xb]$ is a closed subset of $\prod_{t=1}^{N-1} \I_t[\Pr \Xb]$ where $\I(\mu) := (\I_t(\mu))_t$,
  \end{enumerate}
  because then \itref{it:weaklycompact}, \itref{it:kmalgleichgradigstetig}, \itref{it:PPandPcompactness} and Theorem \ref{thm:relco} imply that $\prod_{t=1}^{N-1} \I_t[K]$ is relatively compact in $\prod_{t=1}^{N-1} \FunP {\Xb^t} {\Pr {\Xb_{t+1}}}$, i.e.\ that there is a compact subset $K'$ of $\prod_{t=1}^{N-1} \FunP {\Xb^t} {\Pr {\Xb_{t+1}}}$ which contains $\I[K]$. $K' \cap \I[\Pr \Xb]$ is still compact by \itref{it:IXclosed} and still contains $\I[K]$, showing that $\I[K]$ is relatively compact in $\I[\Pr \Xb]$.

  Showing \itref{it:IXclosed} is relatively simple.
  For two Polish spaces $\X$ and $\Y$ we define a map
  \begin{align*}
    \und \X \Y & : \Pr {\X \times \Pr \Y} \rightarrow \Pr {\X \times \Y} \\
    \intertext{which sends $\nu \in \Pr {\X \times \Pr \Y}$ to the probability $\nu'$ satisfying}
    \int f \d \nu' & = \iint f(x,y) \d \hat y (y) \d \nu (x, \hat y) \fullstop
  \end{align*}
  $\und \X \Y$ is easily seen to be Lipschitz-continuous with constant $1$ by writing out the definition
  \begin{multline*}
    \D_{\Pr{ \X \times \Pr \Y}}^p(\mu, \nu) = \\ \inf_{\gamma \in \Cpl \mu \nu} \int \D(x_1,x_2)^p + \inf_{\hat \gamma \in \Cpl {\hat y_1} {\hat y_2}} \int \D(y_1,y_2)^p \d \hat\gamma(y_1,y_2) \d \gamma(x_1,\hat y_1, x_2,\hat y_2)
  \end{multline*}
  and employing a measurable selector for the inner transport plans $\hat\gamma$ to create from a transport plan $\gamma$ between $\mu$ and $\nu$ a transport plan between $\und \X \Y (\mu)$ and $\und \X \Y (\nu)$ with the same cost as $\gamma$.

  The set $\I[\Pr \Xb]$ is the preimage of the diagonal $\set{ (\mu)_{t \in \set{ 1 \dots N-1 }}}[\mu \in \Pr \Xb] \subseteq {\Pr \Xb}^{N-1}$ under the map which sends $(\mu_t)_t$ to $\pa{\und {\Xb^t} {\Xb_{t+1}}(\mu_t)}_t$.
  This last map is continuous and the diagonal is closed.

  \itref{it:PPandPcompactness} is a special case of Lemma \ref{lem:undAndRelCo} below.
\end{proof}

\begin{lemma}
  \label{lem:undAndRelCo}
  $K \subseteq \Pr { \X \times \Pr \Y }$ is relatively compact iff $\und \X \Y [K]$ is relatively compact.
\end{lemma}
\begin{proof}
  As $\und \X \Y$ is continuous the implication from left to right is clear.

  The other direction is also not hard using Lemmata \ref{lem:compinPrXtimesY} and \ref{lem:compinPrPrX} below, whose proofs we postpone:

  If $\und \X \Y [K]$ is relatively compact, then $\set{ \marg \mu \X }[\mu \in K] = \set{ \marg \nu \X }[{\nu \in \und \X \Y[K]}]$ is relatively compact.
  $\set{ \marg \mu \Y }[\mu \in K]$ is also relatively compact by Lemma \ref{lem:compinPrPrX} because $\avg \Y \big[ \set{ \marg \mu \Y }[\mu \in K] \big] = \set{\marg \nu \Y}[{\nu \in \und \X \Y [K]}]$ is relatively compact. Therefore by Lemma \ref{lem:compinPrXtimesY} $K$ is compact.
\end{proof}

\begin{lemma}
  \label{lem:compinPrXtimesY}
  Let $K \subseteq \Pr {\X \times \Y}$. $K$ is relatively compact iff $K_\X := \set{\marg \mu \X}[\mu \in K]$ and $K_\Y := \set{\marg \mu \Y}[\mu \in K]$ are relatively compact.
\end{lemma}

In analogy to the above definition of $\und \X \Y$ we define
  \begin{align*}
    \avg \X & : \Pr {\Pr \X} \rightarrow \Pr \X \\
    \int f \d(\avg \X(\mu)) & = \int f(x) \d \nu(x) \d \mu(\nu) \fullstop
 \end{align*}

\begin{lemma}
  \label{lem:compinPrPrX}
  Let $K \subseteq \Pr{ \Pr \X }$. Then $K$ is relatively compact iff $\avg \X[K]$ is relatively compact.
\end{lemma}

Lemmata \ref{lem:undAndRelCo}, \ref{lem:compinPrXtimesY}, and \ref{lem:compinPrPrX} have been proved elsewhere.
Lemma \ref{lem:compinPrXtimesY} is very well known for the weak topology --- i.e.\ in the case where the metrics on the base spaces are bounded. In the current setting the proof is only a little more intricate.
Lemma \ref{lem:compinPrPrX} can be found for example in \cite[p. 178, Ch. II]{Sz91} for the weak topology and in \cite{BaBePa18} for our setting. Lemma \ref{lem:undAndRelCo} is also proved there.
For completeness we also provide their proofs here.

We make use of the following variant of Prokhorov's theorem.
\begin{lemma}
  \label{lem:WpProkhorov}
  Let $\X$ be a Polish metric space, let $x_0 \in \X$ be fixed.
  $K \subseteq \Pr \X$ is relatively compact iff for all $\epsilon > 0$ there is a compact set $L \subseteq \X$ with
  \begin{align*}
    \int_{L^c} 1 + \D(x_0, x)^p \d \mu(x) < \epsilon
  \end{align*}
  for all $\mu \in K$.
\end{lemma}
The integrand above will pop up a few times. Let us fix at this point for each Polish metric space $\X$ we will be talking about a point $x_0 \in \X$, and let us agree to do this in a compatible manner, i.e.\ if $x_0$ is the point we have chosen in $\X$ and $y_0$ is the point we have chosen in $\Y$, in $\X \times \Y$ we will chose $(x_0,y_0)$. Similarly, in $\Pr \X$ we choose $\delta_{x_0}$, the dirac measure at $x_0$.
With this convention, define for any Polish metric space $\X$
\renewcommand{\phi}{\varphi}
\begin{align*}
  \phi_\X(x) := 1 + \D(x_0,x)^p \fullstop
\end{align*}
Note that
\begin{align*}
  \phi_{\X \times \Y}(x,y) & = \phi_\X(x) + \phi_\Y(y) - 1 &
  \phi_{\Pr \X}(\nu) & = \int \phi_\X \d \nu
\end{align*}
\begin{proof}[Proof of Lemma \ref{lem:WpProkhorov}]
  As may be common knowledge, the topology induced by $\Wa$ is equal to the initial topology w.r.t.\ the map $\psi$ which send $\mu \in \Pr \X$ to the measure which as density $\phi_\X$ w.r.t.\ $\mu$, when the target space of finite positive measures is equipped with the weak topology. (This can be found for example in \cite[Definition 6.8 (iv) and Theorem 6.9]{Vi09}.) $\psi$ is injective, and surjective onto the closed set of all finite positive measures $\nu$ satisfying
  \begin{align*}
    \int \frac 1 {\phi_\X(x)} \d \nu(x) = 1 \fullstop
  \end{align*}
  $\Pr \X$ is therefore homeomorphic to this set. Translating Prokhorov's theorem for finite positive measures to $\Pr \X$ via $\psi$ gives that $K \subseteq \Pr \X$ is relatively compact iff
  \begin{enumerate}
    \item $\exists M \in \R_+$ s.t.\ $\int \phi_\X \d \mu < M$ for all $\mu \in K$
    \item $\forall \epsilon > 0$ there is a compact set $L \subseteq \X$ s.t.\ $\int_{L^c} \phi_\X \d \mu < \epsilon$.
  \end{enumerate}
  (1) is redundant because we may apply (2) for $\epsilon = 1$ to find a compact set $L$ s.t.\ $\int_{L^c} \phi_\X \d \mu < 1$. $\phi_\X$ is continuous and therefore bounded on $L$, say by $M'$, so that
  \begin{align*}
    \int \phi_\X \d \mu = \int_{L} \phi_\X \d \mu + \int_{L^c} \phi_\X \d \mu \leq M' + 1 =: M \fullstop
  \end{align*}
\end{proof}

\begin{proof}[Proof of Lemma \ref{lem:compinPrXtimesY}]
  $\mu \mapsto \marg \mu \X$ and $\mu \mapsto \marg \mu \Y$ are continuous, so one direction is clear.

  If $K_\X$ and $K_\Y$ are relatively compact, then for any $\epsilon > 0$ there are compact sets $M \subseteq \X$ and $N \subseteq \Y$ s.t.\
  \begin{align}
    \label{eq:MN}
    \int_{M^c} \phi_\X \d (\marg \mu \X) & < \frac \epsilon 4 & 
    \int_{N^c} \phi_\Y \d (\marg \mu \Y) & < \frac \epsilon 4
  \end{align}
  for all $\mu \in K$.
  Because $\phi_\X, \phi_\Y \geq 1$ we also find compact $\bar M \subseteq \X$, $\bar N \subseteq \Y$ s.t.\
  \begin{align}
    \label{eq:bMN}
    \marg \mu \X (\bar M^c) & \leq \frac 1 {\sup_{N} \phi_\Y} \cdot \frac \epsilon 4 &
    \marg \mu \Y (\bar N^c) & \leq \frac 1 {\sup_{M} \phi_\X} \cdot \frac \epsilon 4 \fullstop
  \end{align}
  We show that for $L := M \times \bar N \cup \bar M \times N$ and for all $\mu \in K$
  \begin{align*}
    \int_{L^c} \phi_{\X \times \Y} \d\mu \leq \epsilon \fullstop
  \end{align*}
  
  $\phi_{\X \times \Y}(x,y) < \phi_\X(x) + \phi_\Y(y)$, so we show
  \begin{align*}
    \int_{L^c} \phi_\X(x) \d\mu(x,y) \leq \frac \epsilon 2 \fullstop
  \end{align*}
  $\int_{L^c} \phi_\Y(y) \d\mu(x,y) \leq \frac \epsilon 2$ will follow by symmetry.

  $L^c \subseteq (M \times \bar N)^c = M^c \times \Y \cup M \times N^c$ and therefore
  \begin{align*}
    \int_{L^c} \phi_\X(x) \d\mu(x,y) \leq \int_{M^c \times \Y} \phi_\X(x) \d\mu(x,y) + \int_{M \times N^c} \phi_\X(x) \d \mu(x,y)
  \end{align*}
  The first summand is $\leq \frac \epsilon 4$ by \eqref{eq:MN}, while the second term is bounded by
  \begin{align*}
    \sup_M \phi_\X \cdot \int_{M \times N^c} 1 \d\mu \leq \sup_M \phi_\X \cdot \mu(\X \times N^c) \leq \frac \epsilon 4
  \end{align*}
  by \eqref{eq:bMN}.
\end{proof}

\begin{proof}[Proof of Lemma \ref{lem:compinPrPrX}]
  The left-to-right direction is again obvious because $\avg \X$ is continuous.

  For the other direction we show that for all $\epsilon > 0$ there is a compact set $N \subseteq \Pr \X$ such that for all $\mu \in K$ we have $\int_{N^c} \phi_{\Pr \X} \d \mu \leq \epsilon$.

  Because $\avg \X[K]$ is relatively compact there is for each $n \in \N_+$ a compact set $L_n \subseteq \X$ such that 
  \begin{align}
    \label{eq:Ln}
    \int_{L_n^c} \phi_\X \d(\avg \X(\mu)) \leq \frac \epsilon 2 \cdot 2^{-n} \fullstop
  \end{align}
  We also find for each $n \in \N_+$ a compact set $M_n \subseteq \X$ such that we even have
  \begin{align}
    \label{eq:Mn}
    \int_{M_n^c} \phi_\X \d(\avg \X(\mu)) \leq \frac \epsilon 2 \cdot \frac 1 {\sup_{L_n} \phi} \cdot \frac 1 n \cdot 2^{-n} \fullstop
  \end{align}

  Define
  \begin{align*}
    N & := \set { \nu \in \Pr \X }[ \smallint_{M_n^c} \phi_\X \d\nu \leq \frac 1 n \,\, \forall n ] \text{ ,}
  \end{align*}
  i.e.\ $N = \bigcap_{n \geq 1} N_n$, where
  \begin{align}
    \label{eq:Nn}
    N_n & := \set { \nu \in \Pr \X }[ \smallint_{M_n^c} \phi_\X \d\nu \leq \frac 1 n ] \fullstop
  \end{align}
  Clearly $N$ is compact, again by Lemma \ref{lem:WpProkhorov}.

  We show that for each $\mu \in K$ and for all $n \geq 1$ we have $\int_{N_n^c} \phi_{\Pr \X} \d\mu \leq \epsilon \cdot 2^{-n}$, because then $\int_{N^c} \phi_{\Pr \X} \d\mu = \int_{\pa{\bigcup_{n \geq 1} N_n^c}} \phi_{\Pr \X} \d\mu \leq \sum_{n \geq 1} \int_{N_n^c} \phi_{\Pr \X} \d\mu \leq \epsilon$.

  \begin{multline*}
    \int_{N_n^c} \phi_{\Pr \X} \d \mu = \int_{N_n^c} \int \phi_\X \d\nu \d\mu(\nu) = 
    \int_{N_n^c} \int_{L_n^c} \phi \d\nu \d\mu(\nu) + \int_{N_n^c} \int_{L_n} \phi \d\nu \d\mu(\nu)
  \end{multline*}
  The first summand is $\leq \frac \epsilon 2 \cdot 2^{-n}$ by \eqref{eq:Ln}.
  The second summand we may bound by
  \begin{align*}
    \sup_{L_n} \phi \cdot \int_{N_n^c} 1 \d\mu(\nu) \leq \sup_{L_n} \phi \cdot n \cdot \int_{N_n^c} \int_{M_n^c} \phi \d\nu \d\mu(\nu) \leq \frac \epsilon 2 \cdot 2^{-n} \fullstop
  \end{align*}
  Here we used first \eqref{eq:Nn} and then \eqref{eq:Mn}.
\end{proof}

\section{Other Applications of the Modulus of Continuity}

In this section we give a new proof for Theorem \ref{thm:dprocont} below. \cite{barbiegupta} gave a different proof for the weak topology, i.e.\ for what in our setting corresponds to the case when the metrics on our base spaces are bounded.

The proof uses Lemma \ref{lem:helper} below, which is also used in the companion paper to this one, \cite{AllTopologiesAreEqual}, as an important ingredient in proving that the information topology of Hellwig is equal to the nested weak topology.

\begin{theorem}
  \label{thm:dprocont}
  Let $\mu \in \Pr { \X \times \Y }$, $\nu \in \FunP {\Y} {\Z}$. Then $\dprod$ is continuous at $(\mu, \nu)$.
\end{theorem}

\begin{lemma}
  \label{lem:helper}
  Let $\mu \in \FunP \X \Y$. For any $\epsilon > 0$ there is a $\delta > 0$ s.t. if
  \begin{align*}
    \nu \in \Pr { \X \times \Y } & \text{ with } \W \mu \nu < \delta \text{ and} \\
    \gamma \in \Cpl \mu \nu      & \text{ with } \D^\X(\gamma) < \delta
  \end{align*}
  then
  \begin{align*}
    \D^\Y(\gamma) < \epsilon \fullstop
  \end{align*}
\end{lemma}
\begin{proof}[Proof of Lemma \ref{lem:helper}]
  By Lemma \ref{lem:FunPCharacterisation} we can find $\delta' > 0$ such that $\moc \mu {\delta'} < \frac \epsilon 2$. Set $\delta := \frac {\delta'} 2 \meet \frac \epsilon 2$.

  Let $\Wa(\mu, \nu) = \Wa(\nu, \mu) < \delta$ and let $\gamma \in \Cpl {\mu} {\nu}$ with $\D^\X(\gamma) < \delta$. The former implies that there is a $\eta \in \Cpl {\nu} {\mu}$ with $\D^\X(\eta) \join \D^\Y(\eta) < \delta$.

  Then $\gamma \mcmp \eta \in \Cpl {\mu} {\mu}$ and $\D^\X(\gamma \mcmp \eta) \leq \D^\X(\gamma) + \D^\X(\eta) < 2 \delta \leq \delta' $. This means that $\gamma \mcmp \eta \in \Sh \mu {\delta'}$ and therefore that $\D^\Y(\gamma \mcmp \eta) < \frac \epsilon 2$.

{\def\unused{}%
  \begin{multline*}
    \D^\Y(\gamma) = \pa{\int \D(y_1,y_2)^p \d\pa{\gamma \dprod \eta}(\unused{x_1},y_1,\unused{x_2},y_2,\unused{x_3},\unused{y_3})}^{\frac 1 p} \leq \\
    \pa{\int \pa{\D(y_1,y_3) + \D(y_3,y_2)}^p \d\pa{\gamma \dprod \eta}(\unused{x_1},y_1,\unused{x_2},y_2,\unused{x_3},\unused{y_3})}^{\frac 1 p} \leq \\
    \D^\Y(\gamma \mcmp \eta) + \D^\Y(\eta) < \frac \epsilon 2 + \delta \leq \epsilon
  \end{multline*}
}%
\end{proof}

\newcommand{\tgamma}{\tilde\gamma}

\begin{proof}[Proof of Theorem \ref{thm:dprocont}]
  Let $\epsilon > 0$. Find by Lemma \ref{lem:helper} $\delta > 0$ s.t.\ for all $\nu'$ with $\Wa(\nu,\nu') < \delta$ and all $\kappa \in \Cpl {\nu} {\nu'}$ satisfying $\D^\Y(\kappa) < \delta$ we have $\D^\Z(\kappa) < \epsilon$.


  Let $\nu' \in \Pr { \Y \times \Z }$ s.t.\ $\Wa(\nu,\nu') < \delta$ and let $\mu \in \Pr { \X \times \Y }$ s.t.\ $\Wa(\mu,\mu') < \delta \meet \epsilon$, witnessed by $\gamma \in \Cpl {\mu} {\mu'}$ with $\D^\X(\gamma) \join \D^\Y(\gamma) < \delta \meet \epsilon$.

  From $\gamma$, $\nu$ and $\nu'$ we may use $\dprod$ twice to define a measure $\chi \in \Cpl {\mu \dprod \nu} {\mu' \dprod \nu'}$ which has marginals as shown in the picture below.
  
  \begin{center}
    \begin{tikzpicture}[%
      mymatrix/.style={matrix of nodes,
      column sep=5em,
      row sep=1em},
      node distance=0em and 0em
      ]
      \matrix[mymatrix] (mx) {%
        $\X$ & $\Y$ & $\Z$ \\
        $\X$ & $\Y$ & $\Z$ \\
      };
      \node[fit=(mx-1-1) (mx-1-2), draw, rounded corners, inner sep=0em] (m) {};
      \node[above=-0.5em of m, fill=white, inner sep=0] (mu) {$\mu$};
      \node[fit=(mx-2-1) (mx-2-2), draw, rounded corners, inner sep=0em] (m') {};
      \node[below=-0.5em of m', fill=white, inner sep=0] (mu') {$\mu'$};
      \node[fit=(mx-1-1) (mx-1-2) (mx-2-1) (mx-2-2) (mu) (mu'), draw, rounded corners, inner xsep=.9em, inner ysep=1em] (g) {};
      \node[left=-0.5em of g, fill=white, inner sep=0.2em] (gamma) {$\gamma$};
      \node[fit=(mx-1-2) (mx-1-3), draw, rounded corners, inner sep=.3em] (n) {};
      \node[above=-0.4em of n, fill=white, xshift=1em, inner sep=.1em] (nu) {$\nu$};
      \node[fit=(mx-2-2) (mx-2-3), draw, rounded corners, inner sep=.3em] (n') {};
      \node[below=-0.5em of n', fill=white, xshift=1em, inner sep=.1em] (nu') {$\nu'$};
    \end{tikzpicture}
  \end{center}

  In other words, with $(\nu_y)_y$ a disintegration of $\nu$ w.r.t.\ $\Y$, and similarly for $\nu'$, 
  \begin{align*}
    \int f \d\chi = \iiint f(x,y,z,x',y',z') \d\nu'_{y'}(z') \d\nu_y(z) \d\gamma(x,y,x',y') \fullstop
  \end{align*}

  Setting $\kappa := \marg \chi {\Y \times \Z \times \Y \times \Z}$ we have $\D^\Y(\kappa) = \D^\Y(\gamma) < \delta$, and by our choice of $\delta$ and $\nu'$, $\D^\Z(\kappa) \leq \epsilon$.
  Now $\D^\X(\chi) = \D^\X(\gamma)$, $\D^\Y(\chi) = \D^\Y(\gamma)$, $\D^\Z(\chi) = \D^\Z(\kappa)$ and therefore
  \begin{align*}
    \Wa(\mu \dprod \mu', \nu \dprod \nu') \leq \D^\X(\chi) + \D^\Y(\chi) + \D^\Z(\chi) \leq 3 \epsilon \fullstop
  \end{align*}
\end{proof}

\bibliographystyle{abbrv}
\bibliography{lib/own,lib/joint_biblio}{}

\end{document}